\documentclass{elsarticle}
 \usepackage{graphicx,amssymb,amstext,amsmath}
 \usepackage{bigstrut}
 \usepackage{enumerate}
 \usepackage{amsthm}
 \usepackage[english]{babel}
 \usepackage[all]{xy}
\usepackage{framed}
\usepackage{amsmath}
\usepackage{caption}

 \newtheorem{theorem}{Theorem}[section]
 \newtheorem{thm}[theorem]{Theorem}
 
 \newtheorem{lem}[theorem]{Lemma}
 \newtheorem{cor}[theorem]{Corollary}
 
 \newtheorem{prop}[theorem]{Proposition}
 
\newcounter{remark}[section]
\newenvironment{remark}[1][]{\refstepcounter{remark}\par\medskip
\noindent\textbf{Remark~\theremark #1.} \rmfamily}{\medskip}

\newcommand{\ord}{\mathrm{ord}_2}

\newcommand{\ki}{{k\choose i}}
\newcommand{\kki}{{k-1\choose i}}
\newcommand{\mm}{\;(\mathrm{mod}\,4)}
\newcommand{\mtwo}{\;(\mathrm{mod}\,2)}

\newcommand{\h}{\mathcal{H}}
\newcommand{\ii}{\mathcal{I}}

\makeatletter
\def\@author#1{\g@addto@macro\elsauthors{\normalsize%
    \def\baselinestretch{1}%
    \upshape\authorsep#1\unskip\textsuperscript{%
      \ifx\@fnmark\@empty\else\unskip\sep\@fnmark\let\sep=,\fi
      \ifx\@corref\@empty\else\unskip\sep\@corref\let\sep=,\fi
      }%
    \def\authorsep{\unskip,\space}%
    \global\let\@fnmark\@empty
    \global\let\@corref\@empty  
    \global\let\sep\@empty}%
    \@eadauthor={#1}
}
\makeatother

\begin{document}

\title{Perfect quantum  state transfer on the Johnson scheme}

\author{Bahman Ahmadi\corref{cor1}}
\ead{bahman.ahmadi@shirazu.ac.ir}
\cortext[cor1]{Corresponding author}
  \author{M. H. Shirdareh Haghighi}
 \ead{shirdareh@shirazu.ac.ir}
  \author{Ahmad Mokhtar}
 \ead{ahmad99mokhtar@gmail.com}

\address{Department of Mathematics, Shiraz University, Shiraz, Iran}

 \begin{abstract}
For any graph $X$ with the adjacency matrix $A$, the transition matrix of the continuous-time quantum walk at time $t$ is given by  the matrix-valued function   $\h_X(t)=\mathrm{e}^{itA}$. We say that there is  perfect state transfer in $X$ from the vertex $u$ to the vertex $v$ at time $\tau$ if $|\h_X(\tau)_{u,v}| = 1$. It is an important problem to determine whether perfect state transfers can happen on a given family of graphs. In this paper we characterize all the graphs in the Johnson scheme which have this property. Indeed, we show that the Kneser graph $K(2k,k)$ is the only class in the scheme which admits perfect state transfers. We also show that, under some conditions, some of the unions of the graphs in the Johnson scheme admit perfect state transfer.
 \end{abstract}


\maketitle

\section{\bf Introduction}\label{introduction}

Let  $X$ be a  simple graph. The transition matrix of the continuous-time quantum walk  at time $t$ on the graph $X$ is given by  the matrix-valued function   $\h_X(t)=\mathrm{e}^{itA}$, where $A=A(X)$ is the adjacency matrix of $X$. We say that there is  \textit{perfect state transfer} (or PST) in $X$ from the vertex $u$ to the vertex $v$ at time $\tau$ if $|\h_X(\tau)_{u,v}| = 1$ or, equivalently, if $\h_X(\tau) e_u= \lambda e_v$, for some  $\lambda \in \mathbb{C}$, where $e_u$ is the characteristic vector of the vertex $u$ in $V(X)$. (Note that since $\h_X(\tau)$ is unitary, we must have $|\lambda|=1$.) If the graph $X$ is clear from the context, we may drop the subscript $X$.  We, further, say that $X$ is \textit{periodic} at a vertex $u$ if there is a positive time $\tau$ such that $|\h(\tau)_{u,u}| = 1$, and we say $X$ itself is periodic if there is a positive time $\tau$ such that $|\h(\tau)_{u,u}| = 1$ for all vertices $u$. Note that if there is  PST from $u$ to $v$, then there is also  PST from $v$ to $u$ (so we may just say there is  PST ``between'' $u$ and $v$), and that $X$ will be periodic at $u$ and at $v$ (see \cite{godsil2011periodic}).

 It is  an important question to ask on which graphs PST can happen. As an easy example we can show that PST occurs between the vertices of the the path $P_2$  (see Corollary~\ref{K_2}). We can, also, show directly (see \cite{godsil2012state}) that PST happens between the end-points of the path $P_3$, as well; but this is not the case for any other path. In other words, it has been shown in  \cite{christandl2004perfect} that if $n>3$, then there is no PST on the path $P_n$. On the other hand, in the works such as  \cite{cheung2011perfect,christandl2004perfect,godsil2011periodic,Petkovic}, the existence of PST in different graphs has been shown. See \cite{godsil2012state} for a survey on the subject. 

In \cite{coutinho2015perfect} necessary and sufficient conditions have been proposed for  graphs in association schemes to have PST and using this result the authors have examined some of the graphs in some association schemes. Since one of the most important association schemes is the Johnson scheme $J(n,k)$, it is quite natural to ask whether  any of the classes of this scheme admits PST. 


In this paper,  we prove  that except for the Kneser graph $K(2k,k)$, which is one of the classes of the Johnson scheme for the case $n=2k$, no other class has the capability to have PST. We, further, deduce that the union of all the graphs in the scheme $J(2k,k)$, except the Kneser graph $K(2k,k)$, admits PST as long as the number of vertices of the scheme is divisible by $4$.

In the next section, we will provide  some brief background on the topic. Then we will prove the main result in  Section~\ref{johnson}. In Section~\ref{section_sums} we will study the case of sums of the classes of the scheme $J(2k,k)$ and, finally, conclude the paper with some further discussions  in Section~\ref{conclusion}.


\section{Background}
In this section, we provide the reader with some necessary tools and facts on the topic, which will be used in Section~\ref{johnson}. For the notation and basic facts of  algebraic  graph theory and association schemes  the reader may refer to \cite{godsil2013algebraic} and \cite{eiichi1984algebraic}. 
All the association schemes are assumed to be symmetric. Let  $\mathcal{A}=\{I=A_0, A_1,\ldots, A_d \}$ (where $A_i$ are $v\times v$ matrices) be an association scheme with $d$ classes whose minimal  idempotents (i.e. projections to the eigenspaces) are $\mathcal{E}=\{E_0, E_1,\ldots, E_d\}$. 
Since $\mathcal{E}$ is a  basis for the Bose-Mesner algebra of $\mathcal{A}$, there are coefficients $p_i(j)$ such that 
\[ A_i=\sum_{j=0}^d p_i(j) E_j,\quad \text{for }\, i=0,\ldots, d.\]
For any $i=0,\ldots,d$, the (real) numbers $p_i(j)$, $j=0,\ldots, d$,  are the eigenvalues of $A_i$ and the columns of $E_j$ are eigenvectors of $A_i$. We can think of   $A_i$ as the adjacency matrix of a regular graph $X_i$  on $v$ vertices of valency $p_i(0)$. We will denote the valency of $X_i$ by $v_i$.
On the other hand, the  idempotents $E_i$ can also be written as a linear combination of the $A_j$'s; that is, there are coefficients $q_i(j)$ such that 
\begin{equation}\label{idems_in_terms_of_adjacencies}
E_i=\frac{1}{v}\sum_{j=0}^d q_i(j) A_j,\quad \text{for }\, i=0,\ldots, d.
\end{equation}
The numbers $q_i(j)$ are called the \textit{dual eigenvalues} of the scheme. We note that $q_j(0) = \mathrm{tr}(E_j)$ which is equal to the dimension of the $j$-th eigenspace of the matrices of the scheme which will be denoted by $m_j$. The following fundamental result will be used in Section~\ref{johnson} (see~\cite{eiichi1984algebraic}). 
\begin{prop}\label{dual_evals}
Using the notation above, for any  $i,j=0,\ldots,d$, we have $q_j(i)/m_j=p_i(j)/v_i$.\qed
\end{prop}

Given an association scheme, the following result gives necessary conditions for a graph belonging to the scheme to have PST \cite{coutinho2015perfect}.

\begin{thm} \label{perm_matrix_godsil}
Let $X$ be a graph that belongs to an association scheme with $d$ classes and with adjacency matrix $A =A(X)$. If $X$ admits perfect state transfer at time $\tau$, then there is a permutation matrix $T$ with no fixed points and of order two such that $\h(\tau) =\lambda T$ for some $\lambda \in \mathbb{C}$. Moreover, $T$ is a class of the scheme. \qed
\end{thm}

The following result from~\cite{godsil2012state}, as well, provides a necessary condition for the existence of  PST on a vertex-transitive graph.

\begin{thm}\label{perm_matrix_vtransitive}
Let $X$ be a connected vertex-transitive graph. If $X$ admits perfect state transfer at time $\tau$, then there is a permutation matrix $T$ with no fixed points and of order two such that $\h(\tau) =\lambda T$ for some $\lambda \in \mathbb{C}$.\qed
\end{thm}

Suppose in a class of an association scheme PST occurs and that $E_0,\ldots, E_d$ are the idempotents of the scheme. By Theorem~\ref{perm_matrix_godsil}, $T$ must be one of the classes of the scheme and since the eigenvalues of   $T$  are $\{-1,+1\}$, for any $j=0,\ldots,d $, we have $T E_j=\pm E_j$. As   in     \cite{coutinho2015perfect}, we define a partition $(\ii^+_T, \ii^-_T)$ of $\{0,\ldots, d\}$ as $j\in \ii^+_T$ if $TE_j=E_j$ and $j\in \ii^-_T$ otherwise. We might drop the subscript $T$ if $T$ is clear from the context.
Furthermore, for any $x\in \mathbb{Z}$ we denote by $\mathrm{ord}_2(x)$ the exponent of $2$ in the factorization of $x$ and by convention we define $\mathrm{ord}_2(0)=+\infty$. 

The following theorem, which is  a part of  \cite[Theorem~4.3]{coutinho2015perfect}, provides other necessary conditions for the existence of PST on a graph belonging to an association scheme. 
According to \cite[Lemma~4.2]{coutinho2015perfect}, if PST happens in a graph belonging to an association scheme, then all the eigenvalues of the graph are integers. Let $X$ be a graph in an association scheme on which PST occurs and let  $\lambda_0>\cdots>\lambda_d$ be the distinct integer eigenvalues of $X$. Define
\begin{equation}\label{alpha}
 \alpha_X=\mathrm{gcd}(\{\lambda_0-\lambda_j\; | \; j\in \{0,\ldots,d\}\}). 
 \end{equation}
We will drop the subscript $X$, if the graph $X$ is clear from the context.

\begin{thm} \label{class_of_scheme_pst}
With the notation above, if PST occurs on $X$, then we have $\mathrm{ord}_2(\lambda_0-\lambda_j)>\mathrm{ord}_2(\alpha)$ for $j \in \ii^+$ and $\mathrm{ord}_2(\lambda_0-\lambda_{\ell})=\mathrm{ord}_2(\alpha)$ for $\ell \in \ii^-$.  \qed
\end{thm}

It is clear that for any such $T$, $0 \in \ii^+_T$, but identifying the whole partition  $(\ii^+, \ii^-)$ is not easy in general. 
We next point out the following observation which follows easily from the definition (see \cite{godsil2012state}).

\begin{cor}\label{K_2}
There is PST between the vertices of the graph $K_{2}$ at time $\tau=\pi/2$.\qed
\end{cor}

From \cite{godsil2012state} we also recall the following result   which provides a sufficient condition for a graph to be periodic.
\begin{thm} \label{periodic_godsil}
If all   eigenvalues of a graph $X$ are integers, then $X$ is periodic. \qed
\end{thm}

We conclude this section with a brief introduction to  the Johnson scheme. Throughout the paper we will assume that $n\geq 2k$. For any $i=0,1,\ldots, k$, we define the graph $J(n,k,i)$ to be the graph whose vertex set is the set of all $k$-subsets of $\{1,\ldots,n\}$ and in which two vertices $A$ and $B$ are adjacent if  $|A\cap B|=i$  (in some texts the adjacency in $J(n,k,i)$ is defined as when $|A\cap B|=k-i$, but this makes no difference in our discussion). It can be shown that the set $\mathcal{A}=\{A_0,A_1,\ldots, A_k\}$, where $A_i$ is $A(J(n,k,k-i))$, the adjacency matrix of the graph $J(n,k,k-i)$, is an association scheme (see for example \cite{eiichi1984algebraic}). This scheme is called the \textit{Johnson scheme}, denoted by $J(n,k)$, and the classes $J(n,k,i)$ are called the \textit{generalized Johnson graphs}.
The special cases of $J(n,k,0)$ and $J(n,k,k-1)$ are called the \textit{Kneser graph} (often denoted by $K(n,k)$), and the \textit{Johnson graph}, respectively. Note that the adjacency matrix of  $J(n,k,k)$  is the identity matrix; that is $A_0=I$. In the following lemma, we observe that most of the generalized Johnson graphs are connected.

\begin{lem}\label{connected}
The graph $J(n,k,i)$ is connected if the non-negative integer  $i$   satisfies $2k-n < i < k$.
\end{lem}
\begin{proof}
Consider a vertex $A=\{a_1,\ldots,a_k\}$ from the graph $X=J(n,k,i)$. For any vertex $B=\{b, a_2,\ldots, a_k\}$ which is obtained by changing one of the elements of $A$, we claim that there is a path from $A$ to $B$. According to the conditions on $n,k$ and $i$,  the set $\{1,\ldots,n\}\setminus \{a_1,\ldots, a_k,b\}$ contains at least $k-i$ elements such as $c_1,\ldots, c_{k-i}$; so we have the following path in~$X$:
\[ A\sim \{c_1,\ldots, c_{k-i}, a_{k-i+1}, \ldots, a_k\}\sim \{b,a_2,\ldots, a_{k-i}, a_{k-i+1}, \ldots, a_k\}=B, \]
and the claim follows. If there is a path between two vertices $V$ and $W$, we will write $V\longleftrightarrow W$. Now we show that, for each vertex $C=\{b_1,\ldots,b_k\}$, $A\longleftrightarrow C$. Indeed, if $A\cap C=\varnothing$ the following path is in~$X$:
\[ A \longleftrightarrow \{b_1, a_2,\ldots, a_k\} \longleftrightarrow \{b_1,b_2, a_3,\ldots, a_k\} \longleftrightarrow \cdots \longleftrightarrow \{b_1,b_2,\ldots, b_k\}=C. \]
On the other hand, if $|A\cap C|=m>0$, then without loss of generality we may assume that 
$A\cap C=\{b_1,\ldots, b_m\}=\{a_1,\ldots,a_m\}$. Then, since 
\[\{ a_1,\ldots,a_m,b_{m+1}, a_{m+2},\ldots, a_k\} = \{b_1,\ldots,b_m, b_{m+1},a_{m+2},\ldots, a_k\},\]
the following path exists:
\begin{align*}
 A &\longleftrightarrow \{b_1,\ldots, b_m,b_{m+1}, a_{m+2}, \ldots, a_k\} \\
& \longleftrightarrow \{b_1,\ldots, b_m, b_{m+1}, b_{m+2}, a_{m+3},\ldots, a_k \} \longleftrightarrow\cdots \longleftrightarrow \{b_1,\ldots, b_k\}=C.
 \end{align*}
Therefore the result follows.
\end{proof}
In particular, if $k\geq 2$ and $0<i<k$, then  all the generalized Johnson graphs $J(2k,k,i)$ are connected.

It is clear that the generalized Johnson graph $J(n,k,i)$ is a regular graph with $v={n \choose k}$ vertices and  valency $v_i={k\choose i}{n-k \choose k-i}$. Furthermore, it has been shown (see \cite[p. 220]{eiichi1984algebraic}) that the   eigenvalues of the graphs $J(n,k,i)$  are given by 
\begin{equation}\label{evals_of_johnson_scheme} 
\lambda_j^i= \sum_{\ell=0}^{k-i} (-1)^\ell {j \choose \ell} {k-j \choose k-i-\ell} {n-k-j \choose k-i-\ell}, \quad \text{ for } j=0,\ldots, k,
\end{equation}
with some appropriate multiplicities.
Note that the eigenvalues of the scheme $\mathcal{A}$ are, then, given by $p_i(j) = \lambda_j^{k-i}$. Assuming that $i$ is clear from the context, we will just write $\lambda_j$ instead of $\lambda_j^i$. An immediate consequence of~(\ref{evals_of_johnson_scheme})  is that  all the eigenvalues of the graph $J(n,k,i)$ are integers. Hence, according to Theorem~\ref{periodic_godsil}, we observe the following.
\begin{cor}\label{periodic}
All   graphs in the Johnson scheme are periodic. \qed
\end{cor}

\section{PST on  the Johnson Scheme}\label{johnson}
In this section we will prove the main result of the paper; i.e. we address the question  of whether perfect state transfer can occur in any class of the Johnson scheme. 
First, we  provide a necessary condition for the existence of PST in the Johnson scheme.

\begin{prop}\label{n=2k}
If a graph in the Johnson scheme $J(n,k)$ admits PST then we must have $n=2k$.
\end{prop}
\begin{proof}
If there is PST in  $X$, according to Theorem~\ref{perm_matrix_godsil}, there is  a graph $J(n,k,i)$ in the scheme such that $A(J(n,k,i))=T$ is a permutation matrix of order two; hence $v={n\choose k}$ must be even and one of the graphs of the scheme must be the union of $ \frac{1}{2}{n\choose k}$ copies of $K_2$ edges, that is, $J(n,k,i)\cong  \frac{1}{2}{n\choose k} K_2$. This implies that the valency
\[ {k\choose i}{n-k \choose k-i} \]
equals $1$. Hence, ${k\choose i}=1={n-k \choose k-i}$. The first equality implies $i=k$ or $i=0$, but if the former happens, the graph $J(n,k,i)$ will be the empty graph which implies that $T$ is the identity matrix, a contradiction. Thus we can only have $i=0$. Now the second equality implies $n=2k$.
\end{proof}
\begin{cor}\label{kneser_has_pst}
The Kneser graph $K(n,k)$ admits PST if and only if $n=2k$.
\end{cor} 
\begin{proof}
If $K(n,k)=J(n,k,0)$ has PST, then  according to Proposition~\ref{n=2k}, $n=2k$. On the other hand, since $K(2k,k)\cong \frac{1}{2}{2k\choose k} K_2$ and  there is PST on $K_2$ (Corollary~\ref{K_2}), the result follows.
\end{proof}

In the rest of this section we will show that the Kneser graph $J(2k,k,0)$ is the only graph in the Johnson scheme $J(2k,k)$ which has PST. In other words, we prove the following theorem, which is the main result of the paper.

\begin{thm}\label{main_johnson}
There is PST in the class $J(n,k,i)$ of the Johnson scheme $J(n,k)$ if and only if $n=2k$ and $i=0$.
\end{thm}

In order to prove  Theorem~\ref{main_johnson}, we   need several preliminary lemmas.
First note that, using~({\ref{evals_of_johnson_scheme}), we can list all the eigenvalues of $J(2k,k,i)$.
\begin{lem}\label{evals_of_J(2k,k,i)}
For any $i=0,\ldots,k$, the eigenvalues of $J(2k,k,i)$ are given by
\[
\lambda_j= \sum_{\ell=0}^{k-i} (-1)^\ell {j \choose \ell} {k-j \choose i-j+\ell}^2, \quad  j=0,\ldots, k.  \qed
\]
\end{lem}

In some special cases, we can simplify the formula of the eigenvalues as follows. 
\begin{align}
&\lambda_0 =\ki^2,\label{lambda_0}\\[2pt]
&\lambda_1 =\ki^2 - 2{k \choose i}{k-1\choose i},\label{lambda_1}\\[2pt]
&\lambda_k =(-1)^{k-i}{k\choose i},\label{lambda_k}\\[2pt]
&\lambda_{k-1}=(-1)^{k-i}\left[{k\choose i} - 2{k-1 \choose i}\right], \label{lambda_k_1}\\[2pt]
&\lambda_i=\sum_{\ell=0}^{k-i} (-1)^\ell {i \choose \ell} {k-i \choose \ell}^2,\label{lambda_i}\\[2pt]
&\lambda_{i+1}=\sum_{\ell=0}^{k-i} (-1)^\ell {i+1 \choose \ell} {k-i-1 \choose \ell-1}^2.\label{lambda_i+1}
\end{align}
In particular, in the scheme $\mathcal{A}$, we have $v_i=\ki^2$. Next, we identify the partition $(\ii^-,\ii^+)$ for the   generalized Johnson graph $J(2k,k,i)$.

\begin{prop}\label{evals_partition_even_odd}
For any $i=1,\ldots,k-1$ and any eigenvalue $\lambda_j$ of $J(2k,k,i)$, we have $j\in\ii^+$ if $j$ is even and $j\in\ii^-$ if $j$ is odd.
\end{prop}
\begin{proof}
Let $\{E_0,\ldots,E_k\}$ be the idempotents of the scheme $\mathcal{A}$ and $T=A_k=A(J(2k,k,0))$ be the adjacency matrix of the Kneser graph $J(2k,k,0)$. Let $i=1,\ldots, k-1$ be arbitrary. Recall that $\lambda_j=p_{k-i}(j)$.  We sort the vertices of the scheme in the lexicographic order of the $k$-subsets of $\{1,\ldots, 2k\}$, so that the first vertex is $U=\{1,\ldots,k\}$ and the last (i.e. the $v$-th) vertex is $\overline{U}$ so $T$ is anti-diagonal. As mentioned before Theorem~\ref{class_of_scheme_pst}, for any $j=0,\ldots,k$, we have $TE_j=\pm E_j$. This means that for any $j$, $(TE_j)_{1,1}=\pm (E_j)_{1,1}$. On the other hand, we have $(TE_j)_{1,1}=(E_j)_{1,v}$. Thus $j\in \ii^+$ if $(E_j)_{1,v}= (E_j)_{1,1}$  and $j\in \ii^-$ otherwise. Note that $(E_j)_{1,1}$ is positive (indeed the diagonal entries of these idempotent are always positive). Now, using~(\ref{idems_in_terms_of_adjacencies}) and  Proposition~\ref{dual_evals}, we have   
\[ E_j = \frac{1}{v}\sum_{\ell=0}^k q_j(\ell) A_\ell= \frac{m_j}{v}\sum_{\ell=0}^k \frac{p_\ell(j)}{v_\ell} A_\ell. \]
Since $(A_k)_{1,v}=1$, we deduce  that 
\[ (E_j)_{1,v} =  \frac{m_j}{v} \frac{p_k(j)}{v_k} (A_k)_{1,v}= \frac{m_j}{v} p_k(j). \]
But $p_k(j)$ are the eigenvalues of the Kneser graph and according to Lemma~\ref{evals_of_J(2k,k,i)},
\[ p_k(j)=\lambda_j=(-1)^j.\]
Therefore $(E_j)_{1,v}$ is positive if and only if $j$ is even and the result follows.
\end{proof}

 We first   consider the problem of existence of PST in the case where ${k \choose i}$ is even.
 \begin{lem}\label{no_pst_when_degree_odd}
 For any $i=1,\ldots,k-1$, if ${k\choose i}$ is even, then there is no PST on $J(2k,k,i)$.
 \end{lem}
 \begin{proof}
 Using~(\ref{lambda_1}), we have   $\lambda_1={k \choose i}^2 - 2{k \choose i}{k-1\choose i}$ and according  to Proposition~\ref{evals_partition_even_odd},  $1\in \ii^-$ and 
\begin{equation}\label{lambda_k-lambda_k_1}
\lambda_0-\lambda_1={k\choose i}^2 - {k \choose i}^2 +2{k \choose i}{k-1\choose i}= 2{k \choose i}{k-1\choose i}.
 \end{equation}
If $\ord(\alpha)=t$, then for any $j=1,\ldots, k$, we must have $2^t| \lambda_0-\lambda_j$. In particular, $2^t | \lambda_0-\lambda_k$ and $2^t | \lambda_0-\lambda_{k-1}$. Hence,   $2^t | \lambda_k-\lambda_{k-1}$. That is, $t\leq \ord(\lambda_k-\lambda_{k-1})$. On the other hand, using~(\ref{lambda_k}) and~(\ref{lambda_k_1}) we have 
\[
 \lambda_k-\lambda_{k-1} = 2(-1)^{k-i}  {k-1\choose i};
\]
and since ${k\choose i}$ is even, we have $\ord(\lambda_k-\lambda_{k-1})<\ord(\lambda_0- \lambda_1)$. This means that $\ord(\alpha)<\ord(\lambda_0- \lambda_1)$. According to Theorem~\ref{class_of_scheme_pst}, this shows that $J(2k,k,i)$ cannot have PST.
 \end{proof}
 
To complete the proof, therefore, we should consider the  case where ${k \choose i}$ is odd. To this goal, we will need to make use of Lucas's famous  Theorem (see, for example~\cite{fine1947binomial}).
\begin{thm}\label{lucas_general_p}
Suppose $a$ and $b$ are non-negative integers and $p$ is a prime. Also let 
\[ a=a_rp^r + a_{r-1}p^{r-1} + \cdots+a_1 p + a_0\;\; \text{and} \;\; b=b_rp^r + b_{r-1}p^{r-1} + \cdots+b_1 p + b_0 \]
be the base-$p$ representation of $a$ and $b$, respectively. Then
\[ {a\choose b} \equiv \prod_{\ell=0}^r {a_\ell \choose b_\ell} \;\; (\mathrm{mod}\, p).\qed \]
\end{thm}
For the case $p=2$, we get the following consequence of Theorem~\ref{lucas_general_p}.  We will write  $b\preccurlyeq a$ if  every digit of binary (i.e. base-$2$) representation of $b$ is less than or equal to the corresponding digit of $a$.
\begin{thm}\label{lucas_2}
For any non-negative $a$ and $b$, ${a\choose b}$ is odd if and only if $b\preccurlyeq a$.\qed
\end{thm}

The following is an easy consequence of Theorem~\ref{lucas_2}.
\begin{cor}\label{lucas_cor}
If $a$ is even and $b$  is odd, then ${a\choose b}$ is even.
\end{cor}
 In what follows, we assume, always, that $i\in\{1,2,\ldots,k-1\}$.
 \begin{lem}\label{either_i_or_k-i_choose_l_is_even}
 If $\ki$ is odd, then, for any $\ell=1,\ldots, k-i$, either ${i\choose \ell}$ or ${k-i\choose \ell}$ is even.
 \end{lem} 
 \begin{proof}
Suppose ${i\choose \ell}$ and ${k-i\choose \ell}$ are both odd. Then, according to Theorem~\ref{lucas_2}, we  have $\ell\preccurlyeq i$ and $\ell\preccurlyeq k-i$. Since $\ell\neq 0$, there is a digit $\ell_j$ of the binary representation of $\ell$ such that $\ell_j=1$. Thus we must have $i_j=1$ and $(k-i)_j=1$. On the other hand, since $\ki$ is odd, we have $i\preccurlyeq k$, which implies that $i_j\leq k_j$, thus $k_j=1$. But then we have $(k-i)_j=k_j-i_j=1-1=0$, which is a contradiction. Thus the result follows.
 \end{proof}

 For the results bellow, recall the definition of $\alpha$ from~(\ref{alpha}).

\begin{lem}\label{ord_alpha_when_k_choose_i_is_odd}
If $\ki$ is odd, then $\ord(\alpha) \leq 1$.
\end{lem}
\begin{proof}
Using~(\ref{lambda_i+1}) and Pascal's rule, we have 
\begin{align*}
\lambda_{i+1}&= \sum_{\ell=0}^{k-i} (-1)^\ell {i \choose \ell} {k-i-1 \choose \ell-1}^2 + \sum_{\ell=0}^{k-i} (-1)^\ell {i \choose \ell-1} {k-i-1 \choose \ell-1}^2.
\end{align*}
Since $ {k-i\choose \ell}^2  = {k-i-1 \choose \ell}^2 +  {k-i-1 \choose \ell-1}^2 + 2 {k-i-1 \choose \ell}  {k-i-1 \choose \ell-1}$, we deduce that 

\begin{align*}
\lambda_{i+1}=& \sum_{\ell=0}^{k-i} (-1)^\ell {i \choose \ell} {k-i \choose \ell}^2 
- \sum_{\ell=0}^{k-i} (-1)^\ell {i \choose \ell} {k-i-1 \choose \ell}^2 \\[2pt]
&- 2 \sum_{\ell=0}^{k-i} (-1)^\ell {i \choose \ell} {k-i-1 \choose \ell} {k-i-1 \choose \ell-1} \\[2pt]
&+ \sum_{\ell=0}^{k-i} (-1)^\ell {i \choose \ell-1} {k-i-1 \choose \ell-1}^2\\[4pt]
 = & \lambda_i -  \sum_{\ell=0}^{k-i-1} (-1)^\ell {i \choose \ell} {k-i-1 \choose \ell}^2  + \sum_{\ell=1}^{k-i} (-1)^\ell {i \choose \ell-1} {k-i-1 \choose \ell-1}^2 \\[2pt]
&-2\sum_{\ell=0}^{k-i-1} (-1)^\ell {i \choose \ell} {k-i-1 \choose \ell} {k-i-1 \choose \ell-1}\\[4pt]
 = & \lambda_i -  \sum_{\ell=0}^{k-i-1} (-1)^\ell {i \choose \ell} {k-i-1 \choose \ell}^2  - \sum_{\ell=0}^{k-i-1} (-1)^\ell {i \choose \ell} {k-i-1 \choose \ell}^2 \\[2pt]
&-2\sum_{\ell=0}^{k-i-1} (-1)^\ell {i \choose \ell} {k-i-1 \choose \ell} {k-i-1 \choose \ell-1}\\[4pt]
= & \lambda_i -  2\sum_{\ell=0}^{k-i-1} (-1)^\ell \left[{i \choose \ell} {k-i-1 \choose \ell}^2  + {i \choose \ell} {k-i-1 \choose \ell} {k-i-1 \choose \ell-1}\right]\\[4pt]
= & \lambda_i -  2\sum_{\ell=0}^{k-i-1} (-1)^\ell {i \choose \ell} {k-i-1 \choose \ell} \left[ {k-i-1 \choose \ell} + {k-i-1 \choose \ell-1}\right]\\[4pt]
= & \lambda_i -  2\sum_{\ell=0}^{k-i-1} (-1)^\ell {i \choose \ell} {k-i \choose \ell} {k-i-1 \choose \ell} \\[2pt]
=&  \lambda_i - 2 - 2\sum_{\ell=1}^{k-i-1} (-1)^\ell {i \choose \ell} {k-i \choose \ell} {k-i-1 \choose \ell} \equiv \lambda_i- 2 -0\quad  (\mathrm{mod}\, 4).
\end{align*}
The last line follows from the fact  that according to Lemma~\ref{either_i_or_k-i_choose_l_is_even}, all the summands  in the last summation are even. This shows that $\lambda_i - \lambda_{i+1}\equiv -2\equiv 2$ (mod $4$). Now if $\ord(\alpha)>1$, then  $4 | \lambda_0 - \lambda_i$ and $4| \lambda_0 - \lambda_{i+1}$ and hence  $4 |  \lambda_i - \lambda_{i+1}$, which is a contradiction. Therefore $\ord(\alpha)\leq 1$.
\end{proof}

\begin{lem}\label{no_pst_when_k_choose_i_odd_and_k_1_choose_i_even}
If $\ki$  is odd and  $\kki$ is even, then $J(2k,k,i)$ cannot have PST.
\end{lem}
\begin{proof}
According to~(\ref{lambda_k-lambda_k_1}), we have $\ord(\lambda_0-\lambda_1)\geq 2$; hence using  Lemma~\ref{ord_alpha_when_k_choose_i_is_odd}, Theorem~\ref{class_of_scheme_pst} and the fact that $1\in \ii^-$, the result follows.
\end{proof}

It remains to consider the case where $\ki$ and  ${k-1\choose i}$ are both odd. According to Corollary~\ref{lucas_cor}, these imply that  $i$ is even.  To show the result in this case, we will make use of the following lemmas.

\begin{lem}\label{triple_binom}
If $\ki$ is odd, then for any $\ell=0,1,\ldots, k-i$, the following  product is even:
\[  {k-i \choose \ell}{i\choose k-i-\ell}{i-2\choose k-i-\ell-2}. \]
\end{lem}

\begin{proof}
 Let $s$ be the smallest index for which $i_s$, the $s$-th digit of $i$ in the binary representation, is equal to $1$. As $i\preccurlyeq k$, we  have $k_s=1$; thus $(k-i)_s=0$.  Suppose $ {k-i \choose \ell}$ and ${i\choose k-i-\ell}$ are both odd. Since  $\ell\preccurlyeq k-i$,  it follows that $\ell_s=0$ and the binary representation of $k-i-\ell$ will be 
\[ k-i-\ell=\cdots (k_{s+1}-i_{s+1}-\ell_{s+1})  0 (k_{s-1}-\ell_{s-1})\cdots (k_{0}-\ell_{0}). \] 
Now  $k-i-\ell\preccurlyeq i$ implies, first, that
\[ k-i-\ell=\cdots (k_{s+1}-i_{s+1}-\ell_{s+1})  0\cdots 0. \] 
It, further, implies that for every $t\geq s+1$, if $i_t=0$, then  $k_{t}-i_{t}-\ell_{t}=0$ and if $i_t=1$, then $k_{t}=1$ from which we deduce that $k_{t}-i_t=\ell_t=0$ which, again, implies that $k_{t}-i_{t}-\ell_{t}=0$. Thus $ k-i-\ell=0$ and   ${i-2\choose k-i-\ell-2}=0$; therefore the result follows.
\end{proof}

To prove the next result, we will make use of the well-known Vandermode Convolution identity: let $a,b$ and $c$ be non-negative integers. Then
\begin{equation}\label{vandermonde}
 \sum_{\ell=0}^a {a\choose \ell} {b\choose c-\ell} = {a+b\choose c}.
\end{equation}

\begin{lem}\label{summation_is_odd}
If $\ki$ and $\kki$ are odd and  $k$ is even, then the summation
\[S= \sum_{\substack{\ell=0 \\ \ell : \text{even}}}^{k-i-2}  {k-i \choose \ell} {i-2 \choose k-i-\ell-2}^2  \]
is odd.
\end{lem}
\begin{proof}
Since $k-i$ is even,  according to Corollary~\ref{lucas_cor}, we have $S\equiv S'$ (mod $2$), where
\[S'= \sum_{\ell=0 }^{k-i}  {k-i \choose \ell} {i-2 \choose k-i-\ell-2}^2,   \]
where $\ell$ ranges over all the numbers (even and odd). On the other hand, since for any integer $x$, we have $x\equiv x^2$ (mod $2$), we can write
\begin{align*}
S'&\equiv  \sum_{\ell=0 }^{k-i}  {k-i \choose \ell} {i-2 \choose k-i-\ell-2}  \mtwo.
\end{align*}
 Then using~(\ref{vandermonde}) we deduce that 
\[S'\equiv {k-2\choose i}\mtwo.\]
On the other hand, by the assumption that $\kki$ is odd  and using Corollary~\ref{lucas_cor}, we have 
\[ {k-2\choose i}=\kki - {k-2\choose i-1} \equiv 1 \mtwo.  \]
Therefore the result follows.
\end{proof}

\begin{lem}\label{lambda_{k-i+2}-lambda_{k-i}}
 If $\ki$ and $\kki$ are both odd and $k$ is even, then 
\[\lambda_{k-i+2}-\lambda_{k-i}\equiv 2 \quad (\mathrm{mod} \;4). \]
\end{lem}
\begin{proof}
Recall that the assumptions imply that $i$ is even. We start by writing
\begin{align*}
\lambda_{k-i+2}-\lambda_{k-i}&= \sum_{\ell=0}^{k-i} (-1)^\ell\left[ {k-i+2 \choose \ell} {i-2 \choose k-i-\ell}^2 - {k-i \choose \ell} {i \choose k-i-\ell}^2\right].
\end{align*}
It is obvious (noting Theorem~\ref{lucas_2}) that  the summands   of this summation, for any odd $\ell$, are divisible by $4$; thus we may write 
\begin{align*}
\lambda_{k-i+2}-\lambda_{k-i}&\equiv  \sum_{\substack{\ell=0 \\ \ell : \text{even}}}^{k-i}  {k-i+2 \choose \ell} {i-2 \choose k-i-\ell}^2 - {k-i \choose \ell} {i \choose k-i-\ell}^2 \mm,
\end{align*}
Using Pascal's rule and Corollary~\ref{lucas_cor}, for any such $\ell$, we can write
\begin{align*}
&{k-i+2 \choose \ell} {i-2 \choose k-i-\ell}^2  \equiv {k-i\choose \ell} {i-2 \choose k-i-\ell}^2 +   {k-i\choose \ell-2} {i-2 \choose k-i-\ell}^2  \\[3pt]
& \equiv {k-i\choose \ell} {i \choose k-i-\ell}^2 +   {k-i\choose \ell} {i-2 \choose k-i-\ell-2}^2\\[3pt]
& - 2 {k-i\choose \ell} {i \choose k-i-\ell}{i-2 \choose k-i-\ell-2} +{k-i\choose \ell-2} {i-2 \choose k-i-\ell}^2  \mm. \\
\end{align*}
Therefore, using Lemma~\ref{triple_binom}, we can write 
\begin{align*}
&\lambda_{k-i+2}-\lambda_{k-i}\equiv  \sum_{\substack{\ell=0 \\ \ell : \text{even}}}^{k-i}  {k-i \choose \ell} {i-2 \choose k-i-\ell-2}^2  + {k-i \choose \ell-2} {i-2 \choose k-i-\ell}^2 \\[3pt]
& \equiv  \sum_{\substack{\ell=0 \\ \ell : \text{even}} }^{k-i-2}  {k-i \choose \ell} {i-2 \choose k-i-\ell-2}^2  + \sum_{\substack{\ell=2 \\ \ell : \text{even}} }^{k-i} {k-i \choose \ell-2} {i-2 \choose k-i-\ell}^2 \\[3pt]
& \equiv  \sum_{\substack{\ell=0 \\ \ell : \text{even}} }^{k-i-2}  {k-i \choose \ell} {i-2 \choose k-i-\ell-2}^2  + \sum_{\substack{\ell=0 \\ \ell : \text{even}} }^{k-i-2} {k-i \choose \ell} {i-2 \choose k-i-\ell-2}^2\\[3pt]
& \equiv 2\, \sum_{\substack{\ell=0 \\ \ell : \text{even}}}^{k-i-2}  {k-i \choose \ell} {i-2 \choose k-i-\ell-2}^2   \mm,
\end{align*}
 This, according to Lemma~\ref{summation_is_odd}, completes the proof.
\end{proof}

Using a similar approach, one can  prove the following result.
\begin{lem}\label{lambda_{k-i+1}-lambda_{k-i-1}}
 If $\ki$, $\kki$ and $k$ are odd, then 
\[\lambda_{k-i+1}-\lambda_{k-i-1}\equiv 2 \quad (\mathrm{mod} \;4). \qed\]
\end{lem}

Now we are ready to prove the result in the remaining case.

\begin{lem}\label{no_pst_when_k_choose_i_and_k_1_choose_i_are_odd}
If $\ki$  and  $\kki$ are both odd, then $J(2k,k,i)$ cannot have PST.
\end{lem}
\begin{proof}
Suppose there is PST on $J(2k,k,i)$. Then, according to Theorem~\ref{class_of_scheme_pst}, we must have $\ord(\alpha)=\ord(\lambda_0-\lambda_1)$, which implies, using~(\ref{lambda_k-lambda_k_1}), that $\ord(\alpha)=1$. 
 If $k$ is even, then according to Proposition~\ref{evals_partition_even_odd},  we have $k-i+2, k-i\in \ii^+$ and, according to Theorem~\ref{class_of_scheme_pst}, it must be the case that  $\ord(\lambda_0-\lambda_{k-i+2})>1$ and $\ord(\lambda_0-\lambda_{k-i})>1$. But then we must have $4|\lambda_{k-i+2}-\lambda_{k-i} $ which contradicts Lemma~\ref{lambda_{k-i+2}-lambda_{k-i}}. Also, if $k$ is odd, then $k-i+1, k-i-1 \in \ii^+$ and we get a similar contradiction with Lemma~\ref{lambda_{k-i+1}-lambda_{k-i-1}}. This completes the proof. 
 \end{proof}

\begin{remark}\label{computer}
In order to get the desired ``contradiction'' based on the eigenvalue differences in Lemmas~\ref{ord_alpha_when_k_choose_i_is_odd},~\ref{lambda_{k-i+2}-lambda_{k-i}} and \ref{lambda_{k-i+1}-lambda_{k-i-1}}, we have picked some ``very special'' eigenvalues. Indeed, these eigenvalues are suggested by some computer calculations. It is interesting that the eigenvalue differences rarely result in a contradiction; that is, for majority of the eigenvalues of the graphs, the ``$\ord$'' condition, actually, holds.
\end{remark}

Now we can summarize the proof of the main result as follows.
\begin{proof}[Proof of Theorem~\ref{main_johnson}]
The proof follows from Proposition~\ref{n=2k}, Corollary~\ref{kneser_has_pst} and Lemmas~\ref{no_pst_when_degree_odd} and \ref{no_pst_when_k_choose_i_odd_and_k_1_choose_i_even}, \ref{no_pst_when_k_choose_i_and_k_1_choose_i_are_odd}.
\end{proof}

\section{PST on Sums of the Classes}\label{section_sums}
In this section, we consider sums of the classes of the Johnson scheme $J(n,k)$ and investigate the possibility of existence of PST on the graphs obtained by the unions of some of the generalized Johnson graphs. 
We first recall the following two results (see \cite{godsil2012state}).
\begin{prop}\label{complement}
Suppose $X$ is a regular graph  with PST from a vertex $u$ to another vertex $v$ at time $\tau$. If $\tau$ is an integer multiple of $2\pi/|V(X)|$, then there is PST from $u$ to $v$ at time $\tau$ in the complement graph $\overline{X}$ of $X$.
\qed
\end{prop}

\begin{prop}\label{kay}
If there is PST in a graph $X$ from $u$ to $v$ and also from $u$ to $w$, then $u=w$.\qed
\end{prop}
\begin{prop}\label{automorphism}
Let $u$ and $v$ be two  vertices of a graph $X$. If there is an automorphism $\sigma$ of $X$ such that $\sigma(u)=u$ and $\sigma(v)=w\neq v$, then there cannot be PST between $u$ and $v$.
\end{prop}
\begin{proof}
There is a permutation matrix $P$ corresponding to the automorphism $\sigma$ such that $P^TAP=A$ and the condition on $\sigma$ implies that 
\[ (P^TAP)_{u,v}=A_{u,w}. \]
Since there is PST between $u$ and $v$, there is a time $\tau$ at which
$ |\h(\tau)_{u,v}|=1$,
Now since 
\[ \h(\tau)_{u,v}=\left(\mathrm{e}^{i \tau A}\right)_{u,v}= \left(\mathrm{e}^{i \tau P^TAP}\right)_{u,v}=\left(P^T \mathrm{e}^{i \tau A} P\right)_{u,v}=\left( P^T \h(\tau) P\right)_{u,v}= \h(\tau)_{u,w}; \]
thus
\[ |\h(\tau)_{u,w}| =1,  \]
which contradicts Proposition~\ref{kay}; hence  the result follows.
\end{proof}

The following is an immediate consequence of Proposition~\ref{automorphism}, which has previously been observed in the literature.
\begin{cor}\label{complete_graph}
For any $n>2$, there is no PST in the complete graph $K_n$.\qed
\end{cor}

In order to prove the next result, we make use of the fact that any permutation on $\{1,\ldots,n\}$, naturally induces an automorphism of a single class of the Johnson scheme $J(n,k)$ and, thus, it induces an automorphism of a graph which is the union of any number of the single classes.

\begin{lem}\label{A_B_are_complement}
Let $X$ be a graph which is a union of some of the classes of the Johnson scheme $J(n,k)$. If there is PST between two vertices $A$ and $B$ of $X$, then the $k$-subsets $A$ and $B$ are complement of each other in $\{1,\ldots,n\}$.
\end{lem}
\begin{proof}
Suppose $A$ and $B$ are vertices of $X$ which are not complement of each other. Assume, without loss of generality, that $A=\{1,\ldots,k\}$; thus $B\neq \{k+1,\ldots,n\}$. Consider the two possible cases.

\medskip
 \textbf{Case~1:} $A\cap B\neq \varnothing$. In this case, there are elements $x,x'\in A$ such that $x\in A\cap B$ and $x'\in A\setminus B$. Then the transposition $\sigma=(x\;x')\in \mathrm{Sym}(n)$ is an automorphism of $X$ which fixes the vertex $A$ but moves the vertex $B$. Therefore, according to Proposition~\ref{automorphism}, there cannot be   PST between $A$ and $B$.
 
  \textbf{Case~2:} $A\cap B = \varnothing$. In this case, $B\subsetneq \{k+1,\ldots, n\}$. Hence there are elements $y,y'\in \{k+1,\ldots, n\}$ such that $y\in B$ and  $y'\notin B$. Then the transposition $(y\; y')\in \mathrm{Sym}(n)$ is an automorphism of $X$ which fixes $A$ but moves $B$. Then we get a similar contradiction as in Case~1.
\end{proof}

\begin{cor}\label{sum_n=2k}
Let $X$ be a graph which is  a union of some of the classes of the Johnson scheme $J(n,k)$. If there is PST between two vertices of $X$, then  $n=2k$.\qed
\end{cor}

Note that Proposition~\ref{n=2k} is a special case of Corollary~\ref{sum_n=2k} when $X$ is just a single class. 

In order to see what happens when we consider a union  of  some classes in the Johnson scheme $J(2k,k)$, we first look at the initial  case where $k=2$. Note that the scheme $J(4,2)$ consists of the matrices $\{I, R,T\}$ where the $6={4\choose 2}$-vertex graphs corresponding to the matrices $T$ and $R$, are the Kneser graph $K(4,2)\cong 3K_2$ and the graph $K_6\setminus (3K_2)$, respectively.  According to Theorem~\ref{main_johnson}, the graph corresponding to $R$ cannot have PST. Furthermore, the union of these two graph is the complete graph $K_6$ which, according to Corollary~\ref{complete_graph}, does not have PST. Therefore, no union of the graphs in the scheme $J(4,2)$, can have PST (except for the single class of the graph of $T$). However, the situation is a bit different for the next case where $k=3$. We observe the following general fact.

\begin{cor}\label{all_but_T}
Let $X$ be the graph 
\[ X= J(2k,k,1)\cup \cdots\cup J(2k,k,k-1). \] 
If ${2k\choose k}\equiv 0$ (mod $4$), then for any vertex $A$ of $X$, there is PST from $A$ to the vertex $\overline{A}$ at time $\pi/2$.
\end{cor} 
\begin{proof}
First  note that  PST occurs between any vertex $A$ and its complement  $\overline{A}$ in the Kneser graph $K(2k,k)=J(2k,k,0)$ corresponding to the class $T$ of $J(2k,k)$ at time $\pi/2$. Now since $|V(K(2k,k))|=|V(X)|=4a$, for some integer~$a$, we have
\[\frac{\pi}{2} = a \,\frac{2\pi}{|V(x)|}, \]
hence, according to Proposition~\ref{complement}, there is PST between $A$ and $\overline{A}$ on $X=\overline{K(2k,k)}$ at time $\pi/2$, for any vertex $A$.
\end{proof}

For the case $k=3$, we have $|V(J(n,k))|=20$. The scheme has the classes $\{I,R_1,R_2,T\}$, where the   graphs corresponding to the classes $R_1$ and $R_2$ are two $9$-regular isomorphic graphs. According to Corollary~\ref{all_but_T}, on the union of these two graphs, PST  occurs between any vertex $A$ and its complement $\overline{A}$ (at time~$\pi/2$). Furthermore, the union of all the classes, i.e. the complete graph, does not have PST. On the other hand, computer calculations with Maple software, show that there cannot be PST between any vertex $A$ and $\overline{A}$ on the graph whose adjacency matrix is $T+R_1$ (or $T+R_2$). Therefore, the only case where PST can occur on a union of the classes, is where we consider the  union of all the classes except the Kneser graph $K(6,3)$.  Therefore come up with the following important   question. 

\medskip
\noindent\textbf{Question.} Which  unions of the graphs in the Johnson scheme $J(2k,k)$ can have~PST?
\medskip

So far, besides Corollary~\ref{all_but_T},  the most we can say is the following result.

\begin{lem}\label{sums_perm_T}
Let $X$ be a graph which is a union
\[ J(n,k,i_1)\cup \cdots\cup J(n,k,i_s) \]
 of some of the classes of the Johnson scheme $J(n,k)$, where $s>1$. If there is PST on $X$ then $n=2k$  and   there is a permutation matrix $T$ with no fixed points and of order two such that $\h_X(\tau) =\lambda T$ for some $\lambda \in \mathbb{C}$ with $|\lambda|=1$.
\end{lem}
\begin{proof}
Since $X$ is vertex-transitive and,  according to  Lemma~\ref{connected}, it is connected,  the result follows from Theorem~\ref{perm_matrix_vtransitive} and  Corollary~\ref{sum_n=2k}. Note that since $\h_X(\tau)$ is unitary, we must have $|\lambda|=1$.
\end{proof}

\section{\bf Further Work}\label{conclusion}

We considered the problem of existence of PST on the classes of the Johnson scheme $J(n,k)$. It was not hard to see that $n=2k$ is a necessary condition and also to see that the Kneser graph $J(2k,k,0)$, trivially admits PST. Although  the other classes $J(2k,k,i)$ ($1\leq i \leq k-1$), display important  evidences of capability of having PST (see Remark~\ref{computer}, for example), using some eigenvalue analyses, we showed that none of these classes can admit PST. In addition, we questioned in which unions of the graph in the Johnson scheme $J(2k,k)$, PST can occur which seems an interesting problem.

It sounds an interesting problem to consider other schemes as well. For example, given a permutation group $G\leq \mathrm{Sym}(n)$, which classes of the conjugacy class scheme of $G$ can possibly have PST? It is not hard to see that two major necessary conditions for   $G$ to have PST on any of the graphs in its conjugacy class scheme are as follows: (a) any conjugacy class of $G$ must be closed under inversion, and (b) $G$ must  have a singleton conjugacy class other than the identity class. As an example of such a group, we can name the dihedral group $D_{2n}\leq \mathrm{Sym}(n)$, when $n$ is even.  Indeed we can show that there exist PST in some of the classes of the conjugacy class scheme of $D_{2n}$ (where it occurs on $K_2$'s) and the examples show that there are groups other than the dihedral groups which have PST. Therefore it  seems an interesting problem to classify the permutation groups, acting on $\{1,\ldots,n\}$, which have  the (necessary) conditions (a) and (b).

\section*{Acknowledgment}

The first author would like to thank the financial support of the Iranian National Elites' Foundation.


\bigskip\bigskip
\noindent {\bf References}

\begin{thebibliography}{1}

\bibitem{cheung2011perfect}
Wang-Chi Cheung and Chris Godsil.
\newblock Perfect state transfer in cubelike graphs.
\newblock {\em Linear Algebra and its Applications}, 435(10):2468--2474, 2011.

\bibitem{christandl2004perfect}
Matthias Christandl, Nilanjana Datta, Artur Ekert, and Andrew~J Landahl.
\newblock Perfect state transfer in quantum spin networks.
\newblock {\em Physical review letters}, 92(18):187902, 2004.

\bibitem{coutinho2015perfect}
Gabriel Coutinho, Chris Godsil, Krystal Guo, and Fr{\'e}d{\'e}ric Vanhove.
\newblock Perfect state transfer on distance-regular graphs and association
  schemes.
\newblock {\em Linear Algebra and its Applications}, 478:108--130, 2015.

\bibitem{eiichi1984algebraic}
Bannai Eiichi and Ito Tatsuro.
\newblock Algebraic combinatorics i. association schemes.
\newblock {\em Mathematics Lecture Note Series}, 1984.

\bibitem{fine1947binomial}
Nathan~J Fine.
\newblock Binomial coefficients modulo a prime.
\newblock {\em The American Mathematical Monthly}, 54(10):589--592, 1947.

\bibitem{godsil2011periodic}
Chris Godsil.
\newblock Periodic graphs.
\newblock {\em Electron. J. Combin}, 18(1):P23, 2011.

\bibitem{godsil2012state}
Chris Godsil.
\newblock State transfer on graphs.
\newblock {\em Discrete Mathematics}, 312(1):129--147, 2012.

\bibitem{godsil2013algebraic}
Chris Godsil and Gordon~F Royle.
\newblock {\em Algebraic graph theory}, volume 207.
\newblock Springer Science \& Business Media, 2013.

\bibitem{Petkovic}
Marko~D. Petkovi\'{c} and Milan Ba\v{s}i\'{c}.
\newblock Further results on the perfect state transfer in integral circulant
  graphs.
\newblock {\em Computers \& Mathematics with Applications}, 61(2):300 -- 312,
  2011.

\end{thebibliography}
\end{document}